\documentclass[11pt]{amsart}

\usepackage{color}
\usepackage{amsmath}
\usepackage{amsfonts}
\usepackage{mathrsfs}
\usepackage{amsthm}
\usepackage{fullpage}
\usepackage{bbm}
\usepackage[english]{babel}
\usepackage{mathtools}
\usepackage[normalem]{ulem}
\DeclarePairedDelimiter{\ceil}{\lceil}{\rceil}
\newcommand{\R}{\mathbb{R}}
\newcommand{\C}{\mathbb{C}}
\newcommand{\T}{\mathbb{T}}

\newcommand{\Z}{\mathbb{Z}}
\newcommand{\N}{\mathbb{N}}

\newcommand{\diff}{\mathop{}\!\mathrm{d}}

\usepackage[colorlinks=true,urlcolor=blue,linkcolor=blue,citecolor=blue]{hyperref}
\usepackage[capitalize]{cleveref}
\usepackage{pdfpages}
\newtheorem{theorem}{Theorem}[section]

\newtheorem{definition}[theorem]{Definition}
\newtheorem{example}[theorem]{Example}
\newtheorem{proposition}[theorem]{Proposition}
\newtheorem{question}[theorem]{Question}

\newtheorem{remark}[theorem]{Remark}
\newtheorem{lemma}[theorem]{Lemma}

\newtheorem{claim}{Claim}

\theoremstyle{plain}
\newtheorem*{namedthm}{\namedthmname}
\newcounter{namedthm}
\makeatletter
\newenvironment{named}[2]
{\def\namedthmname{#1}
	\refstepcounter{namedthm}
	\namedthm[#2]\def\@currentlabel{#1}}
{\endnamedthm}
\makeatother

\makeatletter
\renewcommand{\tocsection}[3]{%
  \indentlabel{\@ifnotempty{#2}{\bfseries\ignorespaces#1 #2\quad}}\bfseries#3}
\renewcommand{\tocsubsection}[3]{%
  \indentlabel{\@ifnotempty{#2}{\ignorespaces#1 #2\quad}}#3}

\newcommand\@dotsep{4.5}
\def\@tocline#1#2#3#4#5#6#7{\relax
  \ifnum #1>\c@tocdepth 
  \else
    \par \addpenalty\@secpenalty\addvspace{#2}%
    \begingroup \hyphenpenalty\@M
    \@ifempty{#4}{%
      \@tempdima\csname r@tocindent\number#1\endcsname\relax
    }{%
      \@tempdima#4\relax
    }%
    \parindent\z@ \leftskip#3\relax \advance\leftskip\@tempdima\relax
    \rightskip\@pnumwidth plus1em \parfillskip-\@pnumwidth
    #5\leavevmode\hskip-\@tempdima{#6}\nobreak
    \leaders\hbox{$\m@th\mkern \@dotsep mu\hbox{.}\mkern \@dotsep mu$}\hfill
    \nobreak
    \hbox to\@pnumwidth{\@tocpagenum{\ifnum#1=1\bfseries\fi#7}}\par
    \nobreak
    \endgroup
  \fi}
\AtBeginDocument{%
\expandafter\renewcommand\csname r@tocindent0\endcsname{0pt}
}
\def\l@subsection{\@tocline{2}{0pt}{2.5pc}{5pc}{}}
\makeatother

\usepackage{lipsum}

\title{Distinguishing Sets of Strong Recurrence from Van Der Corput Sets}
\date{\today}

\begin{document}
	\bigskip
	\bigskip
	\bigskip
	\bigskip
	
	\maketitle
	\bigskip
	
	\begin{center}
		\textsc{Andreas Mountakis} \bigskip \bigskip \bigskip \bigskip
		
		\textbf{ABSTRACT}
	\end{center}
	
Sets of recurrence, which were introduced by Furstenberg, and van der Corput sets, which
were introduced by Kamae and Mend\'es France, as well as variants thereof, are important 
classes of sets in Ergodic Theory. In this paper, we construct a set of strong recurrence 
which is not a van der Corput set. In particular, this shows that the class of enhanced van 
der Corput sets is a proper subclass of sets of strong recurrence. 
This answers some questions asked by Bergelson and Lesigne.
	
\bigskip \bigskip
	
\tableofcontents

\newpage
	
\section{Introduction}\label{intro}
Throughout this paper, $\N$ stands for the natural numbers excluding zero, 
while $\N_0$ stands for the natural numbers including zero. 
In addition, whenever we consider a measure on the torus $\T$, it is implicit that 
the underlying $\sigma$-algebra is the Borel $\sigma$-algebra on $\T$, so the measure 
is a Borel measure on $\T$.
	
Inspired by the Poincar\'e Recurrence Theorem, Furstenberg gave the following natural 
definition of sets of recurrence:
\begin{definition}\label{def1}
A set $D\subseteq \mathbb{N}$ is called a set of recurrence if for any measure preserving 
system $(X,\mathcal{A},\mu,T)$ and any $A\in \mathcal{A}$ with $\mu(A)>0$, there exists 
$n\in D$ such that $\mu(A\cap T^{-n}A)>0$.
\end{definition}
	
From the Poincar\'e Recurrence theorem one derives that $\N$ is a set of recurrence. 
In fact, one can easily derive from it that for any $k\in \N$, 
the set $k\N=\{kn:n\in \N\}$ is a set of recurrence. 
More generally, if $(n_{i})_{i\in \mathbb{N}}$ is a strictly increasing sequence of natural
numbers, then it is known that the set $S=\{n_{j}-n_{i}: \: i,j\in \mathbb{N},\: i<j\}$
is a set of recurrence. A more sophisticated example is the following (cf. 
{\cite[Proposition 1.22]{Berg-Les}}): if $f$ is a (non-zero) polynomial with coefficients 
in $\N_0$ and zero constant term, then the sets $D_1=\{f(p-1):p\in \mathbb{P}\}$ and 
$D_2=\{f(p+1):p\in \mathbb{P}\}$
are sets of recurrence ($\mathbb{P}$ denotes the set of prime 
numbers).
	
The notion of sets of recurrence has deep connections with combinatorics and number theory,
as illustrated by \cref{equiv.1}. To state this theorem, we need the notion of upper 
density: For a set $S\subseteq \mathbb{N}$, we define its upper density as 
$\overline{d}(S):=\limsup_{n\to \infty} {\frac{1}{n} |S\cap \{1,2,...,n\} }|$.
	
\begin{theorem}[Cf. {\cite[pages 3-4 and Theorem 3.1]{Berg-Les}}]\label{equiv.1}
Let $D\subseteq \mathbb{N}$. Then the following are equivalent:
\begin{enumerate}
    \item D is a set of recurrence.
    \item D is intersective, i.e. for any $S\subseteq \mathbb{N}$ with
    $\overline{d}(S)>0$, there exist $x, y \in S$  such that $x-y \in D$, or
    equivalently, for any $S\subseteq \N$ with $\overline{d}(S)>0$, we have 
    $(S-S)\cap D \neq \emptyset$.
    \item For any $S\subseteq \mathbb{N}$ with $\overline{d}\big(S\big)>0$, there exists
    $n\in D$ such that $\overline{d}\big(S\cap\big(S-n\big)\big)>0$.
    \item For any sequence $\big(u_n\big)_{n\in \mathbb{N}}$ taking values in 
    $\{0,1\}$ and satisfying 
    \[ \lim_{N\to \infty} \frac{1}{N} \sum_{n=1}^{N} u_{n+h} u_{n} =0
    \:\:\text{  for every }h\in D\]
    we have that 
    \[ \lim_{N\to \infty} \frac{1}{N} \sum_{n=1}^{N} u_n=0. \]
\end{enumerate}
\end{theorem}
	
Later, Kamae and Mend\'es France in \cite{Kam-MF} introduced the notion of van der Corput
(vdC) sets in connection with the theory of uniform distribution on $\T$. An equivalent 
definition of vdC sets, given by Ruzsa in \cite{Ruz}, is the following:
	
\begin{definition}\label{def2}
A set $D\subseteq \N$ is a vdC set if the following holds: If $\sigma$ is a non-negative 
finite measure on $\mathbb{T}$ such that the Fourier transform of $\sigma$ vanishes on $D$
(i.e. $\widehat{\sigma}(h):=\int_\T e^{-2\pi i ht}\diff \sigma(t)=0$ for every $h\in D$),
then $\sigma$ is continuous (i.e. $\sigma(\{t\})=0$ for every $t\in \T$).
\end{definition}
	
Van der Corput sets also connect with number theory, as illustrated by the following 
theorem of Bergelson and Lesigne:
\begin{theorem}[Cf. {\cite[Definition 2 and Theorem 1.8]{Berg-Les}}]\label{sun6}
A set $D\subseteq \mathbb{N}$ is a vdC set if and only if the following holds: Whenever a 
sequence $(u_n)_{n\in \mathbb{N}}$ of complex numbers with modulus 1 satisfies that
\[ \lim_{N\to \infty} \frac{1}{N} \sum_{n=1}^{N} u_{n+h} \overline{u_n} =0\:\: 
\text{ for every } h\in D, \]
then we have 
\[ \lim_{N\to \infty} \frac{1}{N} \sum_{n=1}^{N} u_n =0. \]
\end{theorem}
	
Kamae and Mend\'es France derived from the classical van der Corput Lemma 
\cite{vdCorput31} in uniform distribution that $\N$ is a van der Corput set 
(which is why they gave this name to this class of sets). 
In fact, in \cite{Kam-MF} they proved that for any $k\in \N$, 
the set $k\N=\{kn:n\in \N\}$ is a van der Corput set. 
In addition, in the same paper they show that if $(n_{i})_{i\in \mathbb{N}}$ 
is a strictly increasing sequence of natural numbers, 
then the set $S=\{n_{j}-n_{i}: \: i,j\in \mathbb{N},\: i<j\}$ is a van der Corput set.
Moreover, the examples $D_1$ and $D_2$ described above are van der Corput sets 
(see {\cite[Proposition 1.22]{Berg-Les}}).
	
One may observe that the examples of van der Corput sets given above are also examples of 
sets of recurrence. As Kamae and Mend\'es France showed in {\cite[Theorem 2]{Kam-MF}}, 
every vdC set is a set of recurrence. In fact, when showing that a given set is a set of 
recurrence, one often implicitly shows that it also a vdC set. This raises the natural 
question of whether the opposite is true as well, i.e. is every set of recurrence also a vdC
set? It turns out that the answer is negative, as was shown by Bourgain in the following 
theorem:
	
\begin{theorem}[Cf. \cite{Bou}]\label{bthm}
There is a set $R\subseteq \N$ which is a set of recurrence but not a vdC set.
\end{theorem}

Bourgain's argument in the proof of \cref{bthm} is constructive and finitary, and 
it strongly inspired the work presented in this article.

Following \cite{Berg-Les}, a natural strengthening of the notion 
of sets of recurrence is 
the following definition of sets of strong recurrence:
	
\begin{definition}[Cf. {\cite[Definition 5]{Berg-Les}}]
An infinite set $D\subseteq \mathbb{N}$ is a set of strong recurrence if, given any measure
preserving system $(X,\mathcal{A},\mu,T)$ and any set $A\in \mathcal{A}$ with $\mu(A)>0$, 
we have that 
\[ \limsup_{\substack{m\to \infty\\m\in D}} \mu \big(A\cap T^{-m}A\big)>0. \]
\end{definition}
From the definitions, it is obvious that any set of strong recurrence 
is also a set of recurrence. Forest in \cite{Fo} gave an 
example of a set of
recurrence which is not a set of strong recurrence, 
thereby showing that sets of strong recurrence form a proper 
subclass of the class of sets of recurrence.
	
In the same way sets of strong recurrence provide a quantitative strengthening of sets of 
recurrence, following \cite{Berg-Les}, we also have the following quantitative 
strengthening of van der Corput sets:
	
\begin{definition}[Cf. {\cite[Definition 4 and Theorem 2.1]{Berg-Les}}]
An infinite set $D\subseteq \N$ is called an enhanced van der Corput set if it satisfies 
the following:
Whenever $\sigma$ is a non-negative finite measure on $\T$ such that
\[ \lim_{\substack{d\to \infty\\ d\in D}} \widehat{\sigma}(d)=0\]
then $\sigma$ is continuous (i.e. $\sigma(\{t\})=0$ for every $t\in \T$).
\end{definition}
	
Again from the definitions, it is clear that every enhanced vdC set is also a vdC set. On 
the other hand, all the examples of vdC sets given before, serve also as examples of 
enhanced vdC sets. This gives rise to the interesting question of whether the classes of vdC
sets and enhanced vdC sets coincide, or if there is a vdC set which is not an enhanced vdC 
set.
	
Parallel to the case of vdC sets and sets of recurrence, 
it is known that every enhanced vdC set is a set of strong recurrence 
(see {\cite[Proposition 3.5]{Berg-Les}}). Hence we derive that the examples 
of vdC sets given before, which as remarked are also enhanced vdC sets, 
serve as examples of sets of strong recurrence as well. 
In \cite{Berg-Les}, Bergelson and Lesigne asked the following questions:
\begin{question}\label{q1}
Is every set of strong recurrence an enhanced vdC set?
\end{question}
\begin{question}\label{q2}
Is there any inclusion between the collection of sets of strong recurrence and the 
collection of vdC sets?
\end{question}
	
Since Bourgain's proof of \cref{bthm} was rather involved, \cref{q1} was commented as 
``perhaps quite difficult'' by Bergelson and Lesigne. In this paper, we establish the 
following theorem, which addresses these questions:
	
\begin{theorem}\label{mythm}
There is a set $R\subseteq \mathbb{N}$ which is a set of strong recurrence but not a van 
der Corput set.
\end{theorem}
	
This theorem (whose proof will be presented in \cref{mythm proof}) provides a negative 
answer to \cref{q1} (the set $R$ is a set of strong recurrence, 
but not a vdC set, and thus it is not an enhanced vdC set either). 
In addition, it partially answers \cref{q2}, as it shows that not every set of strong 
recurrence is a vdC set. 
It remains an interesting question whether every vdC set is a 
set of strong recurrence.
	
The paper is organized as follows. In \cref{fourier stuff}, we collect some 
background from Fourier Analysis that is needed in this paper. 
\cref{mythm proof} contains the proof of \cref{mythm}. The proof hinges on a finitistic 
result, \cref{wed9}, which is then proved in \cref{finlem proof}.
\bigskip

\noindent
\textbf{Acknowledgements.} Special thanks go to 
my advisor, Joel
Moreira, for his helpful guidance and beneficial 
comments throughout the writing of 
this paper. In addition, the author would like
to thank the referee for 
their valuable comments and suggestions.

The author is supported by the Warwick Mathematics 
Institute Centre for Doctoral Training, and gratefully 
acknowledges funding by University of Warwick’s 
Chancellors' International Scholarship scheme.
	
\section{Tools from Fourier Analysis}\label{fourier stuff}
In this section, we collect notation, terminology and results from Fourier Analysis 
that are needed in later sections. Most of the material is well known and is presented 
for the convenience of the reader. Recall that $\lambda$ denotes the Lebesgue measure 
on $\T$. If $f$ is a function in $L^1\big(\T, \lambda)$, then we denote by 
$f \diff \lambda$ the measure whose Radon-Nikodym derivative 
(with respect to $\lambda$) is $f$.
	
For a function $f\in L^1\big(\T, \lambda)$, its Fourier transform is 
$\widehat{f}: \Z \to \C,\:$  
$\widehat{f}(k)
=
\int_{\T} f(t)e^{-2\pi i kt} \diff \lambda(t).$ 
We record the following classical 
property of the Fourier transform for later use:
If $f,g \in L^2\big(\T, \lambda)$, then $fg\in L^1\big(\T, \lambda)$ and 
\begin{equation}\label{prop}
\text{ for every } k\in \Z,\:\:
\widehat{fg}\:(k)
=
\sum_{m\in \Z} \widehat{f}(m) \widehat{g}(k-m).    
\end{equation}
If $\sigma$ is a bounded complex valued measure on the torus $\T$, then its 
Fourier transform is the function:
$\widehat{\sigma} : \mathbb{Z} \to \mathbb{C}, 
\:\: \widehat{\sigma}(k)=\int_{\T} e^{-2\pi i kt} \diff \sigma(t).$
It is well known (see for example {\cite[Section I.7]{Katz}}) that the Fourier 
transform $\widehat\sigma$ uniquely determines the measure $\sigma$.
	
Recall the notion of weak$^\ast$ convergence of probability measures: 
If $(\mu_n)_{n\in \mathbb{N}}$, $\mu$ are probability measures on $\mathbb{T}$, 
then $\mu_n\to \mu$ in the weak$^\ast$ sense as $n\to \infty$ if for every 
continuous function $f$ on $\T$, we have $\int_\T f \diff \mu_n \to \int_\T f \diff \mu$,
as $n\to \infty$.
	
\begin{remark}\label{rmk1}
For $(\mu_n)_{n\in \mathbb{N}}$, $\mu$ probability measures on $\mathbb{T}$, 
it is known that
$\mu_n\to \mu$ in the weak$^\ast$ sense as $n\to \infty $ is equivalent to 
$\lim_{n\to \infty} \widehat{\mu_n}(k)=\widehat{\mu}(k)$ for every $k\in \mathbb{Z}$.
\end{remark}
	
Let us now recall the definition of the convolution of two measures on $\T$. 
If $\mu,\nu$ are non-negative finite measures on $\T$, their convolution $\mu \ast \nu$ 
is the non-negative finite measure on $\T$ characterised by the property that 
\[ \int_{\T} f(t) \diff (\mu \ast \nu)(t)=\int_{\T^2} f(x+t) \diff \mu (x) 
\diff \nu (t)\:\text{  for every }f \in C(\T).\]
Identifying a function $f$ with the measure $f\diff \lambda$, 
it makes sense to take the convolution of two functions and 
the convolution of a function with a measure. 
The operation of convolution is associative and commutative. 
Furthermore, for all $k\in \Z$, the following identity holds
\begin{equation}\label{prop1}
\widehat{\mu\ast\nu}\:(k)=\widehat{\mu}(k)\widehat{\nu}(k).
\end{equation}
The same identity holds for convolution of functions, as well as for convolution of 
a function with a measure.
	
For $n\in \N$, the classical Dirichlet and Fejer kernels are the functions 
$D_n,\:F_n :\T \to \C$, whose formulas are given by:
\[D_n(t)=\sum_{k=-n}^{n} e^{2\pi i kt}, \qquad F_n(t)=\sum_{k=-n+1}^{n-1} 
\Bigg(1-\frac{|k|}{n}\Bigg)e^{2\pi i kt}=\frac{1}{n}
\Bigg(\frac{\sin(\pi n t)}{\sin(\pi t)}\Bigg)^2.\]

It is not difficult to verify that $0\leq F_n \leq n$ and $F_n (t)F_m(nt)=F_{nm}(t)$.
A property of the Fejer kernel that we will need is the following:
\begin{lemma}\label{ldl2}
Let $R$ and $L$ be positive integers and let $f(t)=\sum_{|m|\leq RL} a_me^{2\pi i mt}$ be 
a trigonometric polynomial. Assume that $f(t)\geq 0$ for all $t\in \T$. Then, 
$f\leq 4Rf \ast F_L$.
\end{lemma}

\begin{proof}
Consider the function $K:=2F_{2RL}-F_{RL}$ on $\T$, where $F_{2RL},F_{RL}$ denote 
the respective Fejer kernels. Calculating explicitly we have that $\widehat{K}(m)=1$ 
for all $|m|\leq RL$. Therefore, by equation \eqref{prop1} we have that 
$\widehat{f}(m)=\widehat{f\ast K}(m)$ for all $m\in \Z$, which implies that 
$f=f\ast K$. 
Observe 
that $K\leq 2F_{2RL}$ and that 
$ \frac{F_{2RL}(t)}{2RF_{L}(t)}
=
\frac{1}{2R}F_{2R}(Lt) \leq 1$.
Using the previous and the non-negativity of $f$ we obtain that  
$f=f\ast K\leq f\ast 2F_{2RL}\leq 4Rf \ast F_{L}$, as desired.
\end{proof}

We will also need the following proposition:
\begin{proposition}\label{polcr}
If a function $f:\N_0 \to \R$ is non-negative, decreasing and convex
\footnote{The definition of convexity is that 
$ f(b)
\leq
\frac{b-a}{c-a} f(c) + \frac{c-b}{c-a} f(a),$ 
whenever $a\leq b \leq c$.},
and satisfies that $f(\ell)=0$ for some $\ell \in \N_0$, 
then the real trigonometric polynomial $s(t)=\sum_{m\in \Z} f(|m|)e^{2\pi i mt}$
is non-negative.
\end{proposition}

\begin{proof}
For $l=0$, $s(t)=0$ for all $t\in \T$ and the conclusion holds trivially. 
Assume that the proposition holds for some $\ell \in \N_0$, and let $f:\N_0 \to \R$ be 
non-negative, decreasing and convex with $f(\ell +1)=0$. Consider the function 
$g:\N_0\to \R,\:\: g(m)=f(m)-f(\ell) (\ell +1)\widehat{F_{\ell+1}}(m)$, where 
$F_{\ell+1}$ denotes the respective Fejer kernel.

Then it is not difficult to show that $g$ is also non-negative, decreasing, convex 
and $g(\ell)=0$. Therefore, $g$ satisfies the conditions of the induction hypothesis, 
and hence the real trigonometric polynomial $r(t)=\sum_{m\in \Z} g(|m|)e^{2\pi i mt}$ 
is non-negative. Recall that $F_{\ell+1}(t)\geq 0$ for all $t\in \T$, so the trigonometric 
polynomial $r(t)+ f(\ell) (\ell +1) F_{\ell+1} (t)$ is real and non-negative. Calculating,
one sees that $s(t)=r(t)+f(\ell)(\ell+1) F_{\ell+1}$, so $s(t)$ is indeed non-negative 
and this concludes the proof of the Proposition.
\end{proof}

If the Fourier transform of an integrable
function $s:\T \to \C$ satisfies $\widehat{s}(m)=0$ for 
all $|m|\geq N$, then $s$ is a trigonometric polynomial of degree $<N$ and 
\begin{equation}\label{lde12}
\frac{1}{N} \sum_{j=0}^{N-1} s\Big(\frac{j}{N}\Big)
=
\widehat{s}(0).
\end{equation}

\begin{remark}\label{lml}
A measure $\mu$ on the torus $\T$ is said to be supported on the $N$-th roots of unity if
$\mu\Big(\T \setminus \Big\{ \frac{k}{N}: k=0,1,...,N-1\Big\}\Big)=0.$ In that case, the 
Fourier transform $\widehat{\mu}$ is periodic with period $N$. Indeed for $k\in \Z$ 
we have
\begin{align*}
\widehat{\mu}(k+N)
&=
\int_{\T} e^{-2\pi i (k+N)t} \diff \mu (t)
=
\sum_{j=0}^{N-1} e^{-2\pi i (k+N)\frac{j}{N}} \mu\Big(\Big\{\frac{j}{N}\Big\}\Big)
=\\
&=
\sum_{j=0}^{N-1} e^{-2\pi i k \frac{j}{N}} \mu\Big(\Big\{\frac{j}{N}\Big\}\Big)
=
\int_{\T} e^{-2\pi i kt} \diff \mu (t)
=
\widehat{\mu}(k).
\end{align*}
\end{remark}

\section{A Set of Strong Recurrence that is not a van der Corput Set}\label{mythm proof}
	
In this section, we prove \cref{mythm} by constructing a set $R$ which is a set 
of strong recurrence, but not a vdC set. 
This construction hinges on a finitistic result, \cref{wed9}, 
which is proved in \cref{finlem proof}. 
The proof of \cref{wed9} is the most technical part of this paper and it is heavily
inspired by the work of Bourgain in \cite{Bou}. 
Before we can state \cref{wed9} and prove \cref{mythm}, 
we need some definitions and lemmas.
	
From now on, given a positive integer $n$, we denote the set $\{0,1,...,n-1\}$ by $[n]$.
	
\begin{definition}
Let $\epsilon>0$ and let $R\subseteq \mathbb{N}$. We say that $R$ is a set of 
$\epsilon$-recurrence if there exists some $n\in \mathbb{N}$ such that for all 
sets  $E\subseteq [n]$ with 
$|E|>\epsilon n$, there is some $r\in R$ such that $E\cap(E-r)\neq \emptyset$.
\end{definition}
	
\begin{definition}
Let $\epsilon>0$ and let $R\subseteq \mathbb{N}$. We say that $R$ is an $\epsilon$-vdC 
set if whenever a probability measure $\mu$ on $\mathbb{T}$ satisfies that 
$\widehat{\mu}(r)=0$ for every $r\in R$, then $\mu(\{0\})\leq \epsilon$.
\end{definition}
	
\begin{example}\label{sat4}
For $0<\epsilon<1$, take an $N\in \mathbb{N}$ with  $N>\frac{1}{\epsilon}$ and consider 
the set $R:=\{1,...,N-1\}$. Then, it is not difficult to show that if $E$ is a subset of
$[N]=\{0,1,...,N-1\}$ with $|E|>\epsilon N$, then there is some $r\in R$ such that 
$E\cap(E-r)\neq \emptyset$. Hence, $R$ is a set of $\epsilon$-recurrence. 
On the other hand, if $\mu$ is a probability measure on $\T$ such that 
$\widehat{\mu}(r)=0$ for every $r\in \{1,...,N-1\}$, then using the fact that the 
Fejer kernel is non-negative, we obtain that 
\begin{equation*}
N\mu(\{0\})= F_N(0) \mu(\{0\})\leq \int_{\T} F_N(t) \diff \mu(t)=\widehat{\mu}(0)=1,
\end{equation*}
which implies that $\mu(\{0\}) \leq \frac{1}{N} < \epsilon $. Therefore, the set $R$ 
is also an $\epsilon$-vdC set.
\end{example}
	
\cref{sat4} shows that while there are no finite sets of recurrence or finite vdC sets, 
for every $\epsilon>0$ there are finite sets of $\epsilon$-recurrence and finite 
$\epsilon$-vdC sets.
	
\begin{proposition}\label{ldpr}
If $R$ is a set of $\epsilon$-recurrence, then there exist arbitrarily large 
$n\in \mathbb{N}$ such that for all sets $E\subseteq [n]$ with 
$|E|>\epsilon n$, there is some $r\in R$ such that $E\cap(E-r)\neq \emptyset$. 
\end{proposition}
	
\begin{proof}
$R$ is a set of $\epsilon$-recurrence, and hence there exists an $m \in \mathbb{N}$ 
such that whenever $E\subseteq [m]$ satisfies that $|E|>\epsilon m$, then there is 
some $r\in R$ such that $E\cap(E-r)\neq \emptyset$.
		
For $k\in \mathbb{N}$ consider the natural number $km$. Let $E\subseteq [km]$ with
$|E|>\epsilon km$. For every $j\in \{0,1,...,k-1\}$, define 
$E_j :=E\cap [jm, (j+1)m -1].$
Then, $E=\bigcup_{j=0}^{k-1} E_j$ and the union is disjoint, which 
implies that $|E|=\sum_{j=0}^{k-1} |E_j|$. 
Since $|E|>\epsilon k m$, there is some $j \in \{0,1,...,k-1\}$ such that 
$|E_{j}|>\epsilon m$. For this $j$, let
\[ H=E_{j}-jm = E\cap [jm, (j+1)m -1] - j m  \]
Then $H\subseteq [m]$ and $|H|=|E_{j}|>\epsilon m$. 
Therefore, by assumption on $m$, there exists some $r\in R$ such that 
$H\cap (H-r)\neq \emptyset$. Take $h\in H\cap (H-r)$ and note that $h,\:h+r \in H$.
		
Now, let $\ell=h+jm$. Then $\ell,\:\ell+r\in E_{j}$, and therefore 
$\ell\in E_{j} \cap (E_{j}-r)$. 
As a result, $E\cap(E-r)\supseteq E_{j} \cap (E_{j}-r)\neq \emptyset$, 
and this concludes the proof of the proposition.
\end{proof}
	
The following theorem is the finitistic analog of \cref{mythm}:
	
\begin{theorem}\label{wed9}
For every $j\in \mathbb{N}$ and for every $\epsilon \in \big(0,\frac{1}{2}\big)$ there 
is a finite set $R_j \subseteq \mathbb{N}$ which is a set of 
$\big(\frac{1}{j}\big)$-recurrence but not an $\epsilon$-vdC set. 
\end{theorem}
	
Note that if $\mu$ is a probability measure on $\T$ and 
$\mu (\{0\})>\frac{1}{2}$, then for every $n\in \N$ we have 
\begin{equation*}
\left| \widehat{\mu}(n) \right|
=
\Bigg|\int_{\T} e^{-2\pi i nt}\diff \mu(t)\Bigg|
\geq 
\mu(\{0\}) - 
\Bigg| \int_{\T \setminus \{0\}} e^{-2\pi i nt} \diff \mu(t)\Bigg|
>
0.
\end{equation*}
Therefore, for $\epsilon \in [\frac{1}{2},1] $, every non-empty subset of $\N$ 
is an $\epsilon$-vdC set. In this sense, \cref{wed9} cannot be improved.
	
\cref{wed9} will be proved in the next section. For now, using it, 
we will prove the following theorem.
	
\begin{theorem}\label{hldt}
For every $\epsilon\in \big(0,\frac{1}{3}\big)$, there is a set $R\subseteq \mathbb{N}$ 
which is a set of strong recurrence but not an $\epsilon$-vdC set. 
\end{theorem}
	
From the definitions, it is clear that since $R$ is not an $\epsilon$-vdC set, 
then it is also not a vdC set. Therefore, \cref{hldt} implies \cref{mythm}. 
	
\begin{proof}[Proof of \cref{hldt}]
Let $\epsilon \in \big(0,\frac{1}{3}\big)$. Take an 
$\epsilon ' \in \big(0,\frac{1}{2}\big)$ such that 
$\frac{\epsilon '}{1+\epsilon '}>\epsilon$. 
For every $i\in \mathbb{N}$, we use \cref{wed9} to find a finite set $Q_{i}$ which is 
a set of $\big(\frac{1}{i}\big)$-recurrence but not an $\epsilon '$-vdC set.
Then there exists a probability measure $\beta_{i}$ on $\T$ such that 
$\widehat{\beta_{i}}(q)=0$ for every $q\in Q_{i} $ and $\beta_{i} (\{ 0\})> \epsilon '$.
		
Since $Q_{i}$ is a set of $\big(\frac{1}{i}\big)$-recurrence, there is a 
$\kappa_{i} \in \mathbb{N}$ satisfying that for all $E\subseteq [\kappa_{i}]$ 
with $|E|>\frac{\kappa_{i}}{i}$, there is some $q\in Q_{i}$ such that 
$E\cap (E-q) \neq \emptyset$. $Q_i$ is finite, so we may assume that 
$Q_i\subseteq \{1,...,\kappa_i\}$ (for otherwise, we can choose a larger $\kappa_i$ 
using \cref{ldpr}).
		
Let $(R_{j})_{j\in \mathbb{N}}$ be a sequence of subsets of $\mathbb{N}$ having 
the following properties:
\begin{itemize}
\item For every $j\in \mathbb{N}$ there exists an $i(j) \in \mathbb{N}$ 
such that $R_{j}=Q_{i(j)}$.
\item For every $i \in \mathbb{N}$ there exist infinitely many $j \in \mathbb{N}$ 
such that $Q_{i}=R_{j}$.
\end{itemize}
An example of a sequence with those properties is the one defined by 
$R_{1}=Q_{1},\:R_{2}=Q_{1},\:R_{3}=Q_{2},\:R_{4}=Q_{1},\:R_{5}=Q_{2},\:R_{6}=Q_{3}$ 
and so on, i.e. the first term of ($R_{j})_{j\in \mathbb{N}}$ is $Q_{1}$, then the next 
two terms are $Q_{1},\:Q_{2}$, then the next three terms are $Q_{1},\:Q_{2},\:Q_{3}$ and 
we continue in the same fashion.
		
For $j\in \N$, and for $i:=i(j)$, let $n_j=\kappa_{i}$ and consider the measures
\begin{equation*}
\mu_{j}=\beta_{i}\:\: 
\text{ and } 
\:\:\nu_{j}=\mu_{j}-\epsilon ' \delta _{0},
\end{equation*}
where $\delta_0$ denotes the Dirac mass at $0\in \T$. Then $(\mu_{j})_{j\in \N}$ 
is a sequence of probability measures on $\T$ satisfying that $\widehat{\mu_{j}}(\ell)=0$ 
for every $\ell\in R_{j}$. In addition, 
$\mu_{j} (\{ 0\})> \epsilon '$, and therefore $(\nu_{j})_{j\in \N}$ is a sequence of 
non-negative measures on $\T$.
		
For every $j\in \N$, let $M_{j}\in \N$ with $M_j>\frac{1}{\epsilon '}n_{j}(n_{j}+1)$. 
We will construct a positive trigonometric polynomial $b_j$ having the following properties:
\begin{itemize}
\item $\widehat{b_{j}}(m)=0$, for $m\in \Z$ with $|m|>M_j$
\item $\widehat{b_{j}}(m)=\widehat{\nu _{j}}(m)=\widehat{\mu_{j}}(m)-\epsilon '$,
for $m\in \Z$ with $0<|m|\leq n_{j}$
\item $\widehat{b_{j}}(0)=1$.
\end{itemize}
		
To do this, first consider the real trigonometric polynomial whose Fourier transform 
is given by
\begin{equation}\label{hld1}
\widehat{a_{j}}(m)=
\begin{cases}
0, \:\text{ if }\: |m|>n_j \\
\widehat{\nu_j}(m)\frac{|m|}{M_j}, \:\text{ if }\: |m|\leq n_j.
\end{cases}
\end{equation}
Note that $\widehat{a_j}(-m)=\overline{\widehat{a_j}(m)}$, and hence $a_j$ is 
indeed real-valued. In addition for every $m\in \Z$,
$\big|\widehat{\nu_{j}}(m)\big|
=
\Big| \int_{\mathbb{T}} e^{-2\pi imt} \diff \nu_{j}(t)\Big|
\leq
\nu_j(\T)
=
1-\epsilon '$, 
and therefore
\begin{equation*}
||a_{j}||_{\infty}
=
\sup_{t\in \T} |a_j(t)|\leq \sum_{m=-n_{j}}^{n_{j}} |\widehat{a_{j}}(m)|
=
\sum_{m=-n_{j}}^{n_{j}} \big|\widehat{\nu_{j}}(m)\big|\: \frac{|m|}{M_{j}} 
<
\sum_{m=-n_{j}}^{n_{j}} \frac{|m|}{M_{j}}
<
\epsilon '.
\end{equation*}
		
This implies that $a_{j}+\epsilon '$ is a positive trigonometric polynomial.
Note that also $F_{M_{j}}\ast \nu_{j}$ is a non-negative trigonometric polynomial, 
and hence the sum
$b_{j}:=a_{j}+\epsilon ' + F_{M_{j}}\ast \nu_{j}$
is a positive trigonometric polynomial. Calculating, one sees that the Fourier transform 
of $b_j$ has the desired properties.
		
Now, inductively we will define two new sequences $(c_{j})_{j\in \mathbb{N}}$, 
$(d_{j})_{j\in \mathbb{N}}$ of trigonometric polynomials and a sequence 
$(N_{j})_{j\in \mathbb{N}}$ of natural numbers. Let $c_{1}:=b_{1}$ and $N_{1}:=M_{1}$. 
For each $j\geq2$ take $N_{j}\in \mathbb{N}$ such that $N_{j}>2(M_{j-1}+1)N_{j-1}$ 
and define 
\[d_{j},\:c_{j}:\T \rightarrow \R,\:\:\: d_{j}(t)=b_{j}(2N_{j}t) \:\:\text{ and } 
\:\: c_{j}(t)=c_{j-1}(t)d_{j}(t)\]
From the definition it is clear that for every $j\in \mathbb{N}$, $\:c_{j}$ and $d_{j}$ 
are positive trigonometric polynomials. Writing $b_j$ as a linear combination of 
characters we see that 
\begin{equation}\label{ld5}
\widehat{b_j}(m)
=
\widehat{d_j}(2N_j m)
\:\text{ for every } m\in \mathbb{Z}\:
\text{ and }
\:\widehat{d_j}(m)
=
0 \text{ when } m \text{ is not divisible by } 2N_j.
\end{equation}
\begin{claim}\label{claim}
For every $j\in \N$, the Fourier transform $\widehat{c_j}$ satisfies the following:
\begin{itemize}
    \item $\widehat{c_{j}}(m)=0$ for $m\in \mathbb{Z}$ with $|m|\geq N_{j+1}$
    \item $\widehat{c_{j}}(m)=\widehat{c_{j+1}}(m)$ for $m\in \mathbb{Z}$ with $|m|<N_{j+1}$
    \item $\widehat{c_j}(0)=1$ for every $j\in \mathbb{N}$.
    \item $\widehat{c_j}(2N_jm)=-\epsilon '$ for every $m\in R_j$.
\end{itemize}
\end{claim}
The proof of \cref{claim} will be presented after we finish the proof of the theorem.
		
Recall that $\lambda$ denotes the Lebesgue measure on $\T$. For every $j\in \N$, $c_{j}$ 
is a non-negative Borel measurable function on $\T$ and 
$\int_{\mathbb{T}} c_{j}(t)\diff \lambda(t)=\widehat{c_{j}}(0)=1$. Consider the sequence 
of probability measures $\sigma_{j}=c_{j}\diff \lambda$, $j\in \mathbb{N}$, on $\T$. 
Since the space of probability measures on $\T$ is weak$^\ast$ compact, we can consider 
a weak$^\ast$ limit point $\sigma$ of $(\sigma_{j})_{j\in \N}$.
Let $\rho:=\frac{1}{1+\epsilon '}\sigma+\frac{\epsilon '}{1+\epsilon '}\delta_{0}$. 
It is clear that $\sigma$, $\rho$ are probability measures on $\T$.
		
Consider the set $R:=\{r\in\mathbb{N}:\widehat{\rho}(r)=0\}$. Since 
$\rho(\{0\})
\geq
\frac{\epsilon '}{1+\epsilon '}>\epsilon$, 
$R$ is not an $\epsilon$-vdC set. 
We claim that the set $R$ contains the set $2N_j R_j$ for all $j\in \N$. 
	
Since $\sigma$ is a weak$^{\ast}$ limit point of $(\sigma_{\ell})_{\ell \in \N}$, 
we may assume that $\sigma_{\ell}\to \sigma$ in the weak$^{\ast}$ sense as 
$\ell \to \infty$ (for otherwise, we can pass to a subsequence).
Fix a $j\in \mathbb{N}$ and let $m\in R_{j}\subseteq \{1,...,n_{j}\}$. 
For every $\ell>j$ we have 
$2N_{j}m
\leq 
2N_{j}n_{j}
<
N_{\ell}$
and therefore 
$\widehat{c_{\ell}}(2N_{j}m)
=
\widehat{c_{\ell-1}}(2N_{j}m)
=...=
\widehat{c_{j}}(2N_{j}m)$.
Combining that with \cref{rmk1}, we obtain that
\begin{equation*}
\widehat{\sigma}(2N_{j}m)
=
\lim_{\ell\to \infty} {\widehat{\sigma _{\ell}}}(2N_{j}m)
=
\widehat{c_{j}}(2N_{j}m)
=
-\epsilon '.
\end{equation*}
As a result, 
$ \widehat{\rho}(2N_{j}m)
=
\frac{1}{1+\epsilon '} \widehat{\sigma}(2N_{j}m) + 
\frac{\epsilon '}{1+\epsilon '} \widehat{\delta_0}(2N_j m)
=
0$,
which in turn implies that $2N_{j}m\in R$.
		
We will now show that the union $\bigcup _{j\in \N} 2N_j R_j$ is a set of 
strong recurrence. Since $R$ contains this union, it will follow that $R$ is  
a set of strong recurrence as well.
Let $(X,\mathcal{A},\mu,S)$ be a measure preserving system and let 
$A\in \mathcal{A}$ with $\mu(A)>0$. Take an $i\in \mathbb{N}$ such that 
$i>\frac{2}{\mu(A)}$. From the construction of the sequence 
$(R_{j})_{j\in \mathbb{N}}$, there exist infinitely many $j$'s such that 
$R_{j}=Q_{i}$ and $n_j=\kappa_i$. Fix such a $j$ and consider 
the measure-preserving transformation $T:=S^{2N_j}$.
		
For every $x\in X$ define the set 
$E(x):=\big\{ n\in [\kappa_i] :\:\: T ^{n} x \in A\big\}$ 
and consider the function 
\[f:X \rightarrow \mathbb{R}_{\geq 0},\:\:
f(x)
=
\frac{1}{\kappa_i}\big|E(x)\big|
=
\frac{1}{\kappa_i} \sum_{n=0}^{\kappa_i-1} \mathbbm{1} _{A} \circ T ^{n} (x).\]
Since $T$ preserves $\mu$, we have
$\int_{X} f(x) \diff \mu(x)
=
\mu(A)>\frac{2}{i}$. Using this it is not difficult to derive that 
\begin{equation}\label{hld10}
\mu (\{x : |E(x)|\geq \frac{\kappa_{i}}{i}\})
=
\mu(\{x:f(x)\geq \frac{1}{i}\})
>
\frac{1}{i}.
\end{equation}
Note that for every $x\in X$, $E(x)$ is a subset of $\{0,1,...,\kappa_i-1\}$, 
so there are $2^{\kappa_i}$ possibilities for $E(x)$. Combining this with 
\eqref{hld10},
we obtain a set $E\subseteq \{0,1,...,\kappa_i-1\}$ such that
\begin{equation*}
\big|E\big|\geq \frac{\kappa_i}{i}\:\:\:
\text{ and }\:\:\:
\mu(\{x: E(x)=E\})>\frac{1}{i2^{\kappa_i}}.
\end{equation*}
Let 
$C:=\{x: E(x)=E\}$.
Recall that $Q_i$ is a set of 
$\big(\frac{1}{i}\big)$-recurrence,
and since 
$\big|E\big|\geq \frac{\kappa_i}{i}$,
there exists some 
$r\in Q_{i}\:(=R_j)$
such that 
$E\cap(E-r)\neq \emptyset$.
Take an $n\in E\cap(E-r)$.
Then $n,n+r\in E$, and therefore for every 
$x\in C$
we have 
$T^{n}x,T^{n+r}x \in A$,
which implies that 
$C\subseteq T^{-n}({A\cap T^{-r}A})$.
As a result,
\[\frac{1}{2^{\kappa_{i}} i}< \mu(C)\leq \mu(T^{-n}({A\cap T^{-r}A}))
=
\mu(A\cap T^{-r}A)=\mu(A\cap S^{-2N_j r}A),\]
and $r\in Q_i=R_j$, whence $2N_{j}r\in 2N_{j}R_{j}\subseteq R$.
	
So finally, we proved the following: 
For every $j\in \N$ satisfying that $R_j=Q_i$, there exists some $r\in R_j$ such that
$\mu(A\cap S^{-2N_j r}A)
>
\frac{1}{2^{\kappa_{i}} i}$.
Since there are infinitely many $j$'s satisfying that $R_j=Q_i$, 
and since $2N_j R_j\subseteq R$ and $N_j\to \infty$ as $j\to \infty$, 
we obtain that 
\begin{equation*}
\limsup_{\substack{r\rightarrow+\infty\\ r\in R}}  
\:
\mu(A\cap S^{-r} A)
\geq
\frac{1}{i 2^{\kappa_{i}} }
>
0.   
\end{equation*}
		
The measure preserving system $(X,\mathcal{A},\mu,S)$ and the set $A\in \mathcal{A}$ 
were arbitrarily chosen, and hence R is a set of strong recurrence. This concludes 
the proof of the theorem.
\end{proof}
	
Now, we will present the proof of \cref{claim}:
\begin{proof}[Proof of \cref{claim}]
First, by induction in $j$, we will show that 
\begin{equation}\label{hld20}
\text{if }\: |m|\geq N_{j+1}, \:\text{ then }\: \widehat{c_{j}}(m)=0.    
\end{equation} 
Let $m\in \mathbb{Z}$ with $|m|\geq N_{2}.$ Then, since $c_1=b_1$ and $N_2>M_{1}$, 
we have that $\widehat{c_{1}}(m)=\widehat{b_{1}}(m)=0$. Assume that \eqref{hld20} holds 
for some $j$ and let $m\in \mathbb{Z}$ with $|m|\geq N_{j+2}$. Then using \eqref{prop} 
and the induction hypothesis we obtain that 
\begin{equation}\label{sun1}
\widehat{c_{j+1}}(m)
=
\widehat{c_{j}d_{j+1}}(m)
=
\sum_{\substack{k\in \mathbb{Z}\\|k|<N_{j+1}}} \widehat{c_{j}}(k)\widehat{d_{j+1}}(m-k). 
\end{equation}
Observe that $b_{j+1}$ is a trigonometric polynomial with `highest power term' 
at most $M_{j+1}$ (this means that $\widehat{b_{j+1}}(\ell)=0$ for every 
$\ell\in \mathbb{Z}$ with $|\ell|>M_{j+1}$). Since $d_{j+1}(t)=b_{j+1}(2N_{j+1}t)$, 
we have that
$d_{j+1}$ is a trigonometric polynomial with `highest power term' at 
most $2N_{j+1}M_{j+1}$, i.e.
$\widehat{d_{j+1}}(\ell)
=
0$  
for every  $\ell\in \mathbb{Z}$ with $|\ell|>2N_{j+1}M_{j+1}$.
For $k\in \mathbb{Z}$ with $|k|<N_{j+1}$, we have
$|m-k|
>
N_{j+2}-N_{j+1}
>
2N_{j+1}M_{j+1}$, and therefore $\widehat{d_{j+1}}(m-k)=0$. As a result, all the terms 
in the sum in \eqref{sun1} are $0$, which in turn implies that 
$\widehat{c_{j+1}}(m)=0$. This concludes the induction.
		
Let $m\in \mathbb{Z}$ with $|m|<N_{j+1}$. Using \eqref{prop} and \eqref{hld20}, 
we obtain that 
\begin{equation}\label{sun2}
\widehat{c_{j+1}}(m)
=
\widehat{c_{j}d_{j+1}}(m)
=
\sum_{\substack{k\in \mathbb{Z}\\|k|<N_{j+1}}}
\widehat{c_{j}}(k)\widehat{d_{j+1}}(m-k)    
\end{equation}
For $k\in \mathbb{Z}$ with $|k|< N_{j+1}$, if $k\neq m$, then $m-k$ is not a multiple 
of $2N_{j+1}$, and hence from \eqref{ld5} we get that $\widehat{d_{j+1}}(m-k)=0$. 
Therefore in \eqref{sun2} 
the only possible non-zero term in the sum is the one corresponding to $k=m$, 
and thus
\begin{equation}\label{sun3}
\widehat{c_{j}}(m)
=
\widehat{c_{j-1}}(m)\widehat{d_j}(0)
=
\widehat{c_{j-1}}(m)\widehat{b_j}(0)
=
\widehat{c_{j-1}}(m).    
\end{equation}
In particular, this also implies that 
$\widehat{c_{j}}(0)=\widehat{c_{j-1}}(0)=...=\widehat{c_{1}}(0)=\widehat{b_{1}}(0)=1$ 
for every $j\in \mathbb{N}$.
		
Finally, let $m\in R_j$. Then applying \eqref{hld20} to $\widehat{c_{j-1}}$, we obtain that
\begin{equation}\label{hld30}
\widehat{c_{j}}(2N_{j}m)
=
\sum_{\substack{k\in \mathbb{Z}\\|k|<N_{j}}} 
\widehat{c_{j-1}}(k)\widehat{d_{j}}(2N_{j}m-k)   
\end{equation}
Observe that the only $k\in \Z$ with $|k|<N_j$ satisfying that $2N_{j}m-k$ is a multiple 
of $2N_{j}$ is $k=0$, and therefore, by \eqref{ld5}, the only non-zero term in the sum 
in \eqref{hld30} is the one corresponding to $k=0$. As a result, 
$\widehat{c_j}(2N_{j}m)
=
\widehat{c_{j-1}}(0)\widehat{d_{j}}(2N_{j}m)
=
\widehat{b_{j}}(m)
=
\widehat{\mu_{j}}(m)-\epsilon '
=
-\epsilon '$, 
where the last equality is due to the fact that $\widehat{\mu_j}=0$ on $R_j$. 
This concludes the proof of \cref{claim}.
\end{proof}
	
While \cref{wed9} states that for every 
$\epsilon \in \big(0,\frac{1}{2}\big)$ and every 
$j\in \N$, there is a set $R$ of $\big(\frac{1}{j}\big)$-recurrence 
which is not an $\epsilon$-vdC set, in \cref{hldt} we managed to prove
the existence of set $R$ of strong recurrence 
which is not an $\epsilon$-vdC set 
only for $\epsilon \in \big(0,\frac{1}{3}\big)$. 
An interesting question that arises naturally is whether \cref{hldt} 
can be extended so that it covers the case 
$\epsilon \in \big(0,\frac{1}{2}\big)$ :
\begin{question}
Is it true that for every $\epsilon \in \big(0,\frac{1}{2}\big)$ there is a 
set $R$ which is a set of strong recurrence but not an $\epsilon$-vdC set?
\end{question}
\section{A Finitistic Theorem}\label{finlem proof}
In this final section, we present the proof of \cref{wed9}. In order to give the proof, 
we first need to state and prove a series of lemmas.
	
\subsection{A Combinatorial Lemma}
	
\begin{lemma}\label{comblem}
Let $j,Q \in \N$, with $Q$ even and sufficiently large depending on $j$, and 
let $P\in \N$ with $P>Q \log 2j^2$. Then, for any 
$E\subseteq [Q^P]=\{0,1,...,Q^P -1\}$ with $\big| E \big| > \frac{Q^P}{j}$, 
there is a point $y\in E-E$ such that, when written in base $Q$, i.e.,
\[y=\sum_{i=0}^{P-1} y_i Q^i,\:\: y_i\in \{0,1,...,Q-1\},\]
one of the digits, say $y_s$, satisfies $\frac{Q}{2} \leq y_s < \frac{Q}{2} + 8j$ 
and all the other digits satisfy that $1\leq y_i < 8j$.
\end{lemma}
	
In order to prove \cref{comblem} we need another two lemmas, which we will state 
and prove right now. 
\begin{remark}
From now on, we will work several times with the group of integers 
$\bmod Q$, i.e. with $\Z_Q=\Z/ Q\Z=\{0,1,...,Q-1\}$. 
On $\Z_Q$ we put the metric $d$ defined by 
$d(m+Q\Z,n+Q\Z)=\min \{|m-n+kQ|:k\in \Z\}$.
\end{remark}
	
\begin{lemma}[Cf. {\cite[Lemma 3.5]{Bou}}]\label{comblem1}
Let $\ell,Q \in \N$ with $Q$ even and let $P\in \N$ with $P>Q \log \ell$. 
Consider the space 
\begin{equation*}
\Z_Q ^P=\{(x_0,x_1,...,x_{P-1}):x_i \in \Z_Q\text{ for }i\in \{0,1,...,P-1\}\}.   
\end{equation*}
Then, for any $B\subseteq \Z_Q ^P$ with $|B|>\frac{Q^P}{\ell}$, 
there is a pair of points $x,x'\in B$ and some integer $s\in \{0,1,...,P-1\}$ such that 
\begin{itemize}
\item $x_i=x_i'$ for $i=0,1,...,s-1$
\item $x_s=0$ and $x_s'=\frac{Q}{2}$
\item $d(x_i,x_i')\leq 2$ for $i=s+1,...,P-1$.
\end{itemize}
\end{lemma}
	
The proof of this lemma is adapted from Bourgain's proof of Lemma 3.5 in \cite{Bou}. 
\begin{proof}
For a set $A\subseteq \{0,1,...,P-1\}$, we define the relation 
$\sim_A$ on $\Z_Q ^P$ as follows:
For $x, y \in \Z_Q ^P$, 
$x\sim_A y :\iff x_i=y_i \text{ for all } i\in A 
\text{ and } d(x_i,y_i)\leq 1 \text{ for all } i\notin A$.
It is obvious that $\sim_A$ is reflexive 
($x\sim_A x$ for all $x\in \Z_Q ^P$) and symmetric ($x\sim_A y \iff y\sim_A x$).
Now, let $B\subseteq \Z_Q^P$ with 
$|B|>\frac{Q^P}{\ell}.$ For every $s\in \{0,1,...,P\}$, 
we denote the set $\{0,1,...,s-1\}$ by $[s]$ (where $[0]=\emptyset$).
Consider the set 
$B_s:=\{y\in \Z_Q^P : y\sim_{[s]} x \text{ for some } x\in B\}$.
Then, $B=B_P\subseteq B_{P-1}\subseteq...\subseteq B_0
\subseteq \Z_Q^P$.\\
\underline{Claim:} For some $s\in\{0,1,...,P-1\}$, there is 
some $y=(y_0,y_1,...,y_P)\in B_{s+1}$ such that replacing the 
$s$ coordinate of $y$ with any $k\in \Z_Q$ we still get a point in 
$B_{s+1}$, i.e. 
\begin{equation*}
y(k):=(y_0,...,y_{s-1},k,y_{s+1},...,y_{P-1})\in B_{s+1} 
\text{ for all } k\in \Z_Q.   
\end{equation*}

\begin{proof}[Proof of Claim]
\renewcommand\qedsymbol{$\triangle$}
We proceed by contradiction. Assume that the claim doesn't hold. 
Then for every $s\in \{0,1,...,P-1\}$ and every $y\in B_{s+1}$ there 
is a $k\in \Z_Q$ such that 
$y(k) \notin B_{s+1}$. 
Since $y=y(y_s)\in B_{s+1}$ we conclude that there exists a $k\in \Z_Q$ 
such that $y(k)\notin B_{s+1}$ but $y(k-1)\in B_{s+1}$. 
Then, $y'=y(k)\in B_{s}\setminus B_{s+1}$ and $y'$ only disagrees with 
$y$ at the $s$ coordinate. So we have that for all $s\in \{0,1,...,P-1\}$ 
and all $y\in B_{s+1}$ there is some $y'\in B_{s}\setminus B_{s+1}$ 
such that $y'$ only disagrees with $y$ at the $s$ coordinate.
			
Fix $s\in \{0,1,...,P-1\}$. Then, according to the above, we can define 
a map $\phi: B_{s+1} \to B_{s} \setminus B_{s+1}$ so that for every $y\in B_{s+1}$,
$\phi (y)$ disagrees with $y$ only at the $s$ coordinate. 
Then 
\begin{equation}\label{lmeq}
|B_{s+1}|=|\phi ^{-1} (B_{s} \setminus B_{s+1})|
=
\Big| \bigcup_{z\in B_{s} \setminus B_{s+1}} \phi^{-1} (\{z\})\Big|
=
\sum_{z\in B_{s} \setminus B_{s+1}} |\phi^{-1} (\{z\})|.
\end{equation}
Let $z \in B_{s} \setminus B_{s+1}$. For each $y \in \phi^{-1} (\{z\})$, we have that 
$y$ and $z=\phi (y)$ disagree only at the $s$ coordinate. Since there are exactly 
$Q-1$ points in $\Z_Q^P$ which disagree with $z$ only at the $s$ coordinate, we 
obtain that $|\phi^{-1} (\{z\})|\leq Q-1$. Using \eqref{lmeq} we get that 
$|B_{s+1}|\leq (Q-1) |B_{s} \setminus B_{s+1}| $, 
or equivalently $|B_{s} \setminus B_{s+1}| \geq \frac{|B_{s+1}|}{Q-1}$, and therefore 
$|B_{s}|\geq \frac{Q}{Q-1}|B_{s+1}|$.

Then, inductively we obtain that
\begin{equation}\label{ldeq1}
|B_0|
\geq
\Bigg(\frac{Q}{Q-1}\Bigg) |B_1|
\geq ...\geq 
\Bigg(\frac{Q}{Q-1}\Bigg)^P |B_P|
=
\Bigg(\frac{Q}{Q-1}\Bigg)^P |B|
>
\Bigg(\frac{Q}{Q-1}\Bigg)^P \frac{Q^P}{\ell}.
\end{equation}
Since $P>Q\log \ell$, from \eqref{ldeq1} we obtain that 
\begin{equation}\label{hhld1}
|B_0|
>
\Bigg( \frac{Q}{Q-1} \Bigg)^{Q\log \ell} \frac{Q^P}{\ell}
=
Q^P \ell ^{Q(\log Q - \log (Q-1))-1}
\geq 
Q^P,
\end{equation}
where the last inequality in \eqref{hhld1} is due to 
$\log Q-\log (Q-1)\geq \frac{1}{Q}$.
On the other hand, $B_0\subseteq \Z_Q ^P$, and therefore 
$|B_0|\leq Q^P$, which is a contradiction according to \eqref{hhld1}.
\end{proof}
Now we are ready to finish the proof of the lemma. 
We use the claim to find an $s\in \{0,1,...,P-1\}$ and a 
$y=(y_0,y_1,...,y_{P-1})\in B_{s+1}$ such that for every $k\in \Z_Q$, 
$(y_0,...,y_{s-1},k,y_{s+1},...,y_{P-1})\in B_{s+1}$. 
Consider the elements 
\[z=(y_0,...,y_{s-1},0,y_{s+1},...,y_{P-1}) \text{ and } 
z'=(y_0,...,y_{s-1},\frac{Q}{2},y_{s+1},...,y_{P-1}),\] 
both of which lie in $B_{s+1}$. Now, take $x,x'\in B$ such that 
$x\sim_{[s+1]} z$ and $x'\sim_{[s+1]} z'$. 
Then, for $i\in \{0,1,...,s-1\}$ we have $x_i=z_i=y_i=z_i '=x_i '$. 
In additon $x_s=z_s=0$, 
$x_s ' =z_s ' =\frac{Q}{2}$ and for 
$i\in \{s+1,...,P-1\} $ we have 
$ d(x_i,x_i')\leq d(x_i,z_i)+d(z_i,z_i')+d(z_i',x_i')\leq 1+0+1=2 $.
\end{proof}
	
The second lemma we need is the following stronger form of the 
Poincar\'e Recurrence Theorem:
\begin{lemma}\label{PRT}
Let $(X,\mathcal{B},\mu, T)$ be a m.p.s. and let $E\in \mathcal{B}$ 
with $\mu(E)>0$. Then for some $n\leq \frac{2}{\mu(E)}$ we have 
$\mu(E\cap T^{-n}E)\geq \frac{\mu(E)^2}{2}$.
\end{lemma}

For sake of completeness, we give the proof of this lemma.
\begin{proof}
For every $n\in \N$ consider the function $f_n=\mathbbm{1}_{T^{-n}E}$. 
Let $\alpha=\mu(E)$ and take $R\in \N$ with $R\geq \frac{2}{\alpha}$.
Using the Cauchy-Schwarz inequality we obtain that
\begin{eqnarray*}
(R\alpha)^2
&=&
\Bigg( \int_X \sum_{n=1}^{R} f_n \diff \mu\Bigg)^2
\leq
\int_X \Bigg(\sum_{n=1}^{R} f_n \Bigg)^2 \diff \mu
=
\sum_{n=1}^{R} \int_X f_n^2 \diff \mu + 2\sum_{\substack{n,m=1\\n<m}}^{R} 
\int_X f_n f_m \diff \mu
=\\
&=&
R\alpha + 2\sum_{\substack{n,m=1\\n<m}}^{R} \mu(T^{-n}E \cap T^{-m}E)
\leq
R\alpha + (R^2-R)\:\: 
\max_{\substack{i,j\in \{1,...,R\} \\ i<j}}\:\:\mu(T^{-j}E \cap T^{-i}E)
\end{eqnarray*}
As a result
\begin{equation}\label{mdeq}
\max_{\substack{i,j\in \{1,...,R\} \\ i<j}}\:\:\mu(T^{-j}E \cap T^{-i}E) 
\geq 
\frac{R\alpha ^2 -\alpha}{R-1}
\geq
\frac{\alpha^2}{2}.   
\end{equation}
Since $R$ was arbitrary, \eqref{mdeq} holds for every 
$R\in \N$ with $R\geq \frac{2}{\alpha}$. Take 
$R_0:=\ceil{\frac{2}{\alpha}}$
($=$the ceiling of $\frac{2}{\alpha}$,
i.e. the smallest integer which is greater than or equal $\frac{2}{\alpha}$). 
Then $R_0$ satisfies \eqref{mdeq}, and hence there exist 
$i,j\in \{1,...,R_0\}$ with $i<j$ 
such that 
$\mu(T^{-j}E\cap T^{-i}E)\geq \frac{\alpha^2}{2}$. 
Taking $n:=j-i$ we have that 
$1\leq n\leq \frac{2}{\alpha}=\frac{2}{\mu(E)}$ 
and 
$\mu(E\cap T^{-n}E)\geq \frac{\alpha^2}{2}=\frac{\mu(E)^2}{2}$.
\end{proof}
	
Now, we are ready to present the proof of \cref{comblem}:
\begin{proof}[Proof of \cref{comblem}:]
Let $j,Q \in \N$, with $Q$ even and sufficiently large depending on $j$, 
and let $P\in \N$ with $P>Q \log 2j^2$. 
In addition, let $E\subseteq [Q^P]=\{0,1,...,Q^P -1\}$ with 
$|E| > \frac{Q^P}{j}$. We will identify $[Q^P]$ 
with the product $X=\Z_Q^P$ through the expansion in base $Q$ digits. 
Give $X$ the normalised counting probability measure $\mu$ 
and consider the measure-preserving transformation
\[T: X\to X,\:\: T(x_0,x_1,...,x_{P-1})=(x_0+4, x_1+4,...,x_{P-1}+4)\]
where the addition is $\bmod Q$.
Then $\mu(E)=\frac{|E|}{Q^P}>\frac{1}{j}$, so by \cref{PRT}, 
there is some $n\in \N$ with $1\leq n\leq \frac{2}{\mu(E)}<2j$ 
such that $\mu\big(T^{-n}E\cap E\big)\geq \frac{\mu(E)^2}{2}$. 
For such an $n$, consider the set $B:=T^{-n}E\cap E$. 
Then $|B|=\mu(B)Q^P>\frac{Q^P}{2j^2}$ and we can apply \cref{comblem1} 
for the set $B$ (and for $\ell=2j^2$) to obtain a pair of points 
$x,x'\in B$ and some integer $s\in \{0,1,...,P-1\}$ such that 
\begin{itemize}
    \item $x_i=x_i'$ for $i\in \{0,1,...,s-1\}$
    \item $x_s=0,\:x_s'=\frac{Q}{2}$
    \item $d(x_i,x_i')\leq 2 $ for $i\in \{s+1,...,P-1\}$
\end{itemize}
Since $x'\in B$, we have that $T^n x'\in E$. Also $x\in B\subseteq E$. \\
\underline{$T^n x' \neq x:$} $T^n x'=(x_0 ' +4n,...,x_{P-1} ' +4n)$ 
and $x=(x_0,...,x_{P-1})$. Since $4n<8j$, we have 
$1\leq x_s ' +4n =\frac{Q}{2} +4n < \frac{Q}{2} +8j$. 
$Q$ is chosen sufficiently large depending on $j$, so we may 
assume that $\frac{Q}{2}+8j<Q$. 
Therefore, we have that $1\leq x_s'+4n<Q$ and $x_s=0$, 
which in particular gives that $x_s \neq x_s'+4n\:\:\bmod Q$. 
As a result therefore $T^nx'\neq x$.\\
Now, we see $T^n x',x$ as elements of $\N$. 
Consider $y=\max \{T^nx'-x,x-T^nx'\}$ where now the subtraction is 
taken in $\Z$. Then $y\in E-E$ and $y>0$. 
Let $y=\sum_{i=0}^{P-1} y_iQ^i$ be the base $Q$ expansion of $y$. 
Because of the possible borrow of digits when 
we perform the subtraction we have that 
$y_i\:\:\bmod Q\in \{
x_i'+4n-x_i\:\:\bmod Q,x_i'+4n-x_i-1\:\:\bmod Q, 
x_i-x_i'-4n\:\:\bmod Q,x_i-x_i'-4n-1\:\:\bmod Q\}.$ 
In every case, since $1\leq n <2j$ we have that 
$\frac{Q}{2}\leq y_s <\frac{Q}{2}+8j$ 
and for $i\neq s$ we have $1\leq y_i<8j$. 
This concludes the proof of the lemma.
\end{proof}
	
\subsection{Some measures on the $N$-th roots of unity}
The following lemma gives rise to some measures that will be useful
in the proof of the \cref{wed9}:
\begin{lemma}[Cf. {\cite[Lemma 5.1]{Bou}}]\label{ldl1}
Let $\ell, Q \in \N$, with $Q$ even and sufficiently large depending on $\ell$. 
Then there is an absolute positive constant 
\footnote{This constant $C$ is independent of $\ell$ and $Q$. 
From the proof one sees that we can choose for example $C=320$.} 
$C$ such that for every $k\in \N$, 
there is a non-negative finite measure $\sigma$ on the torus 
$\T$ satisfying the following:
\begin{itemize}
\item $\sigma$ is supported on $Q^{k+1}$-th roots of unity, so by \cref{lml}
its Fourier transform is periodic with period $Q^{k+1}$
\item $\widehat{\sigma}(0)\leq  1 + C  \frac{\ell^3}{Q^2}$
\item $\widehat{\sigma}(m)=1$ for $m\in \Z$ with 
$1\leq \frac{m}{Q^k}\leq \ell$
\item $\widehat{\sigma}(m)=-1$ for $m\in \Z$ with 
$\frac{Q}{2}\leq \frac{m}{Q^k}\leq \frac{Q}{2}+\ell$. 
\end{itemize}
\end{lemma}
	
\begin{proof}
$Q$ is sufficiently large depending on $\ell$, 
so we may assume that $Q>4\ell$. 
Let $N:=Q^{k+1}$ and consider the real trigonometric polynomials 
\begin{equation*}
p:\T \to \C,\: 
p(x)=\sum_{\substack{m\in \Z\\ |m|\leq \ell Q^k}}  
\Bigg( 1-\cos2\pi \frac{\ell Q^k-|m|}{N} \Bigg) e^{2\pi imx} 
\end{equation*}
\begin{equation*}
r:\T \to \C,\: r(x)=2\cos\big(\ell Q^k2\pi x\big)-\cos 
\Big( \Big(\frac{N}{2}-\ell Q^k\Big)2\pi x\Big) -\cos \Big(\Big(\frac{N}{2}+
\ell Q^k\Big)2\pi x\Big) . 
\end{equation*}
In addition, let $F_{Q^k}$ denote the $Q^{k}$-th Fejer kernel 
and consider the real trigonometric polynomial
\[s(x)=16\ell p\ast F_{Q^k}\:(x) +r(x)p(x).\]
We want to show that $s\geq 0$. Consider the function 
\begin{equation*}
f:\N_0 \to \R,\:\: f(m)=
\begin{cases}
1-\cos2\pi \frac{\ell Q^k-m}{N},\:\:\text{ if }\:0\leq m \leq \ell Q^k\\
0,\:\:\text{ otherwise }
\end{cases}   
\end{equation*}
We have assumed that $4\ell<Q$, 
and hence for every $x\in [0,\ell Q^k]$, we have 
$0\leq \frac{\ell Q^k-x}{N}< \frac{\pi}{2}$.
Then the function $x \mapsto 1- \cos 2\pi \frac{\ell Q^k-x}{N} $ 
is decreasing and convex in the interval $[0,\ell Q^k]$, which 
implies that $f$ is decreasing and convex in $\N_0$. 
Obviously, $f$ is non-negative, and therefore $f$ satisfies the 
conditions in \cref{polcr}. As a result, we 
obtain that the trigonometric polynomial 
$p(x)=\sum_{m\in \Z} f(|m|)e^{2\pi i mx}$ satisfies $p\geq 0$.  
Using \cref{ldl2} we obtain that
$p\leq (4\ell p)\ast F_{Q^k}$ and hence $4p\leq (16 \ell p)\ast F_{Q^k}$.
In addition, $|r(x)|\leq 4$, so we have that 
$-r(x)p(x)\leq |r(x)|p(x)\leq 4p(x)\leq (16 \ell p)\ast F_{Q^k}$, 
and therefore $s(x)\geq 0 \text{ for every } x\in \T$.
		
We can then consider the non-negative finite measure 
$\sigma$ on the torus $\T$, defined by
\begin{equation*}
\sigma=\frac{1}{2}\Big(\delta_{\frac{1}{N}} + \delta_{\frac{-1}{N}}\Big)
+ 
\frac{1}{N}\sum_{n=0}^{N-1} s\Big(\frac{n}{N}\Big)\delta_{\frac{-n}{N}}.
\end{equation*}
$\sigma$ is clearly supported on the $N$-th roots of unity. 
Now we need to check that the Fourier transform of $\sigma$ 
has the desired properties. 
To do that, we first compute the Fourier transforms of $p$ and $r$.
\begin{equation}\label{hhld2}
\widehat{p}(m)= 
\begin{cases}
1-\cos\Big(2\pi \frac{\ell Q^k-|m|}{N}\Big), \text{ if } |m|\leq \ell Q^k\\
0,\text{ otherwise}. \end{cases} 
\end{equation}
On the other hand, for the Fourier transform of $r$ we have
		
\begin{equation}\label{hhld3}
\widehat{r}(m)= 
\begin{cases}
1, \text{ if } |m| = \ell Q^k\\
-\frac{1}{2}, \text{ if } |m| = \frac{N}{2} - \ell Q^k  
\text{ or } |m|= \frac{N}{2} + \ell Q^k\\
0, \text{ otherwise}. \end{cases} 
\end{equation}
	
Next, we will compute the Fourier transform of s. For every $m\in \Z$, we have 
\begin{equation}\label{form}
\widehat{s}(m)
=
16\ell \widehat{p}(m)\widehat{F_{Q^k}}(m)+\sum_{t\in \Z}\widehat{r}(t)\widehat{p}(m-t).
\end{equation}
	
We now split in cases depending on $m\in \Z$.\\
\underline{\textbf{Case 1.}} If $|m|\geq N=Q^{k+1}$:\\
$|m|>Q^k$, so $\widehat{F_{Q^k}}(m)=0$ 
and therefore using \eqref{form} and \eqref{hhld3} we obtain
\begin{eqnarray*}
\widehat{s}(m)&=&\widehat{p}(m-\ell Q^k) 
+
\widehat{p}(m+\ell Q^k) -\frac{1}{2} \widehat{p}(m+\frac{N}{2}-\ell Q^k)\\
& & -\frac{1}{2} \widehat{p}(m+\frac{N}{2}+\ell Q^k) 
-
\frac{1}{2} \widehat{p}(m-\frac{N}{2}-\ell Q^k) 
-\frac{1}{2} \widehat{p}(m-\frac{N}{2}+\ell Q^k).
\end{eqnarray*}
For $Q$ sufficiently large depending on $\ell$, 
all the arguments in the above expression have absolute value greater 
than $\ell Q^k$, and hence from \eqref{hhld2} we obtain 
that $\widehat{s}(m)=0$. So we proved that
\begin{equation}\label{lde1}
\widehat{s}(m)
=
0 
\text{ for every } m\in\Z \text{ with } 
|m|\geq N=Q^{k+1}.
\end{equation}
\underline{\textbf{Case 2.}} If $m=0$:\\
$Q$ is sufficiently large depending on $\ell$, so we may assume that 
$\ell Q^k< \frac{N}{2} -\ell Q^k$. Then using \eqref{form}, \eqref{hhld2} and \eqref{hhld3} 
we obtain 
\begin{equation}\label{lde2}
\widehat{s}(0)
=
16\ell \widehat{p}(0)\widehat{F_{Q^k}}(0)
=
16\ell \Bigg(1-\cos2\pi \frac{\ell}{Q}\Bigg)
\leq
16\ell \times 20 \Bigg(\frac{\ell }{Q}\Bigg)^2
\leq
C \frac{\ell ^3}{Q^2},
\end{equation}
where for the first inequality we use the fact that 
$\text{for every } x\in \R,\:\: 1- \cos(2\pi x)\leq 20x^2$.\\
\underline{\textbf{Case 3.}} If $Q^k\leq |m|\leq \ell Q^k$:\\
$|m|\geq Q^k$, so $\widehat{F_{Q^k}}(m)=0$. 
$Q$ is sufficiently large depending on $\ell$, so we may assume that 
$\frac{Q}{2} -2\ell>\ell$. 
Then, for every $t\in \Z$ with $|t|\geq \frac{N}{2} - \ell Q^k$ we have that 
$|m-t|\geq |t|-|m|\geq \Big(\frac{Q}{2} -2\ell \Big)Q^k>\ell Q^k$, 
which implies that 
$\widehat{p}(m-t)=0$. Using \eqref{form} we then get
	
\begin{equation*}
\widehat{s}(m)=\sum_{\substack{t\in \Z \\ |t|
<
\frac{N}{2}-\ell Q^k}}\widehat{r}(t)\widehat{p}(m-t).    
\end{equation*}
For the $t$'s in the above sum, we know that $\widehat{r}(t)$ 
is non-zero exactly when $t=\ell Q^k$ or $t=-\ell Q^k$.\\
\underline{\textbf{Subcase 3.1.}} If $m>0$, then 
$|m-\ell Q^k|=\ell Q^k-m<\ell$ 
and therefore
\begin{equation*}
\widehat{p}(m-\ell Q^k)
=
1-\cos\Big(2\pi \frac{\ell Q^k-|m-\ell Q^k|}{N}\Big)
=
1-\cos2\pi \frac{m}{N}.
\end{equation*}
On the other hand, $|m+\ell Q^k|>\ell Q^k$ 
and hence 
$\widehat{p}(m+\ell Q^k)=0$. 
Combining all those, we obtain that 
$\widehat{s}(m)=\widehat{p}(m-\ell Q^k)= 1-\cos2\pi \frac{m}{N}$.\\
\underline{\textbf{Subcase 3.2.}} If $m<0$, 
then similarly one shows that 
$\widehat{s}(m)=\widehat{p}(m-\ell Q^k)= 1-\cos2\pi \frac{m}{N}$.\\
So after all we have that 
\begin{equation}\label{lde3}
\widehat{s}(m)
=
1-\cos2\pi \frac{m}{N} 
\text{ for every } m\in \Z \text{ with } 
Q^k\leq |m|\leq \ell Q^k.
\end{equation}
\underline{\textbf{Case 4.}} If $Q^k-N\leq m \leq \ell Q^k-N$:\\
Again, since $Q$ is chosen sufficiently large depending on $\ell$, 
we may assume that
$\ell Q^k-N<-Q^k$ and $2\ell Q^k-\frac{N}{2}<-\ell Q^k$. 
Then $m< -Q^k$, and thus
$\widehat{F_{Q^k}}(m)=0$. Therefore, using \eqref{form} we get that 
$\widehat{s}(m)
=
\sum_{t\in \Z}\widehat{r}(t)\widehat{p}(m-t)$. 
If $t\in \Z$ is such that $\widehat{r}(t)\neq 0$, 
then $-t \leq \frac{N}{2}+\ell Q^k$, so 
$m-t \leq 2\ell Q^k-\frac{N}{2}<-\ell Q^k$ and therefore $\widehat{p}(m-t)=0$. 
As a result we obtain
\begin{equation}\label{lde4}
\widehat{s}(m)
=
0 
\text{ for every } m\in \Z \text{ with } 
Q^k-N\leq m \leq \ell Q^k-N.
\end{equation}
\underline{\textbf{Case 5.}} If $\frac{N}{2}\leq m \leq \frac{N}{2}
+\ell Q^k$:\\
$m\geq \frac{N}{2}>Q^k$, so 
$\widehat{F_{Q^k}}(m)=0$ and using \eqref{form} we have that 
$\widehat{s}(m)=\sum_{t\in \Z}\widehat{r}(t)\widehat{p}(m-t)$. 
The only $t\in \Z$ such that $\widehat{r}(t)\neq 0$ and 
$|m-t| \leq \ell Q^k$ is $t=\frac{N}{2}+\ell Q^k$. 
Therefore using \eqref{hhld2} and \eqref{hhld3} we obtain that 
\begin{equation*}
\widehat{s}(m)
=
-\frac{1}{2} \widehat{p}\Big(m-\frac{N}{2}-\ell Q^k\Big)
=
-\frac{1}{2}\Bigg(1-\cos2\pi\frac{\ell Q^k-|m-N/2-\ell Q^k|}{N}\Bigg).
\end{equation*}
Observe that $|m-N/2-\ell Q^k|=N/2+\ell Q^k-m$. Therefore we showed that 
\begin{equation}\label{lde5}
\widehat{s}(m)
=
-\frac{1}{2}\Bigg(1-\cos2\pi\frac{m-N/2}{N}\Bigg) 
\text{ for every } m\in \Z \text{ with } 
\frac{N}{2}\leq m \leq \frac{N}{2}+\ell Q^k.
\end{equation}
\underline{\textbf{Case 6.}} If $-\frac{N}{2}\leq m \leq -\frac{N}{2}+\ell Q^k$:\\
$Q$ is chosen sufficiently large depending on $\ell$, so we may assume that 
$-\frac{N}{2}+\ell Q^k <-Q^k$. Then $m<-Q^k$, and thus 
$\widehat{F_{Q^k}}(m)=0$. Using \eqref{form} we then obtain 
$\widehat{s}(m)=\sum_{t\in \Z}\widehat{r}(t)\widehat{p}(m-t)$. 
The only  $t\in \Z$ such that $\widehat{r}(t)\neq 0$ and 
$|m-t| \leq Q^k$ is $t=-\frac{N}{2}+\ell Q^k$.
Therefore using \eqref{hhld2} and \eqref{hhld3} we obtain that
\begin{equation*}
\widehat{s}(m)
=
-\frac{1}{2} \widehat{p}\Big(m+\frac{N}{2}-\ell Q^k\Big)
=
-\frac{1}{2}\Bigg(1-\cos2\pi\frac{\ell Q^k-|m+N/2-\ell Q^k|}{N}\Bigg).
\end{equation*}
Observe that $|m+N/2-\ell Q^k|=-N/2+\ell Q^k-m$. Therefore we showed that
\begin{equation}\label{lde6}
\widehat{s}(m)
=
-\frac{1}{2}\Bigg(1-\cos2\pi\frac{m+N/2}{N}\Bigg) 
\text{ for every } m\in \Z \text{ with } 
-\frac{N}{2}\leq m \leq -\frac{N}{2}+\ell Q^k.
\end{equation}

Now we are ready to compute the Fourier transform of the measure $\sigma$.
From the definition of $\sigma$, 
we get that for every $m\in \Z$
\begin{equation}\label{hhld7}
\widehat{\sigma}(m)
=
\frac{1}{2}\Big(\widehat{\delta_{\frac{1}{N}}}(m) + 
\widehat{\delta_{\frac{-1}{N}}}(m)\Big)
+ 
\frac{1}{N}\sum_{n=0}^{N-1} s\Big(\frac{n}{N}\Big)\widehat{\delta_{\frac{-n}{N}}}(m)
=
\cos\Big(2\pi \frac{m}{N}\Big) + \frac{1}{N}\sum_{n=0}^{N-1}
s\Big(\frac{n}{N}\Big) e^{-2\pi im\frac{n}{N}}.
\end{equation}
For $m=0$, using \eqref{lde12} and \eqref{lde2} we get
\begin{equation*}
\widehat{\sigma}(0)
=
1+ \frac{1}{N}\sum_{n=0}^{N-1} s\Big(\frac{n}{N}\Big)
=
1+\widehat{s}(0)
\leq
1+ C \frac{\ell ^3}{Q^2}.
\end{equation*}
Using \eqref{hhld7} and the Fourier inversion formula on $s$, we obtain that
\begin{eqnarray*}
\widehat{\sigma}(m)-\cos\Big(2\pi \frac{m}{N}\Big)
&=& 
\frac{1}{N}\sum_{n=0}^{N-1} \sum_{t\in \Z} \widehat{s}(t)e^{2\pi it 
\frac{n}{N}} e^{-2\pi im\frac{n}{N}}
=
\sum_{t\in \Z} \widehat{s}(t) \frac{1}{N}\sum_{n=0}^{N-1} 
e^{2\pi i (t-m)\frac{n}{N}}
=\\
&=&
\sum_{t\in \Z} \widehat{s}(t+m)\frac{1}{N}\sum_{n=0}^{N-1} 
e^{2\pi i t \frac{n}{N}}
=
\sum_{t\in \Z}\widehat{s}(t N+m).
\end{eqnarray*}
From equation \eqref{lde1} we have that $\widehat{s}(t)=0$ for $|t|\geq N$. 
Thus if $0<m\leq N$, then 
\begin{equation}\label{lde20}
\widehat{\sigma}(m)
=
\cos\Big(2\pi \frac{m}{N}\Big)+\widehat{s}(m)+\widehat{s}(m-N).    
\end{equation}
For $m\in \Z$ with $1\leq m/Q^k\leq \ell $, we have that 
$Q^k\leq m \leq \ell Q^k$ and $ Q^k-N \leq m -N \leq \ell Q^k -N$. 
Therefore from equations \eqref{lde20}, \eqref{lde3} and \eqref{lde4} we have 
\begin{equation*}
\widehat{\sigma}(m)
=
\cos\Big(2\pi \frac{m}{N}\Big)+\Big(1-\cos\Big(2\pi \frac{m}{N}\Big)\Big)+0
=
1.
\end{equation*}
On the other hand, for $m\in \Z$ with $Q/2\leq m/Q^k\leq Q/2+ \ell$
we have
$N/2 \leq m \leq N/2 +\ell Q^k $ and    $-N/2\leq m -N \leq \ell Q^k-N/2$. 
Therefore from equations \eqref{lde20}, \eqref{lde5} and \eqref{lde6} we obtain
\begin{eqnarray*}
\widehat{\sigma}(m)&=&\cos\Big(2\pi \frac{m}{N}\Big)
-
\frac{1}{2}\Big(1-\cos2\pi \frac{m-N/2}{N}\Big)
-\frac{1}{2}\Big(1-\cos2\pi\frac{m-N/2}{N}\Big)=\\
&=&\cos\Big(2\pi\frac{m}{N}\Big)+\cos\Big(2\pi\frac{m}{N}-\pi\Big)-1=-1.
\end{eqnarray*}
\end{proof}
	
\subsection{The proof of \cref{wed9}}
The following remark is also useful:

\begin{remark}\label{ldl}
If $R$ is a set of $\epsilon$-recurrence, then there exists $n\in \N$ such that 
whenever $E\subseteq [n]$ and $|E|>\epsilon n$ we have $(E-E)\cap R \neq \emptyset$.
But we have $(E-E)\cap R \subseteq [n]$, so in fact $(E-E)\cap (R\cap [n])\neq \emptyset$,
which means that $R\cap [n]$ is a set of $\epsilon$-recurrence.
\end{remark}

Finally, we have assembled all the tools that we need in order to prove 
\cref{wed9}. For the convenience of the reader, we state \cref{wed9} again,
before proving it.
	
\begin{named}{\cref{wed9}}{}
For every $j\in \mathbb{N}$ and for every 
$\epsilon \in \big(0,\frac{1}{2}\big)$ 
there is a finite set 
$R_j \subseteq \mathbb{N}$ which is a set of 
$\big(\frac{1}{j}\big)$-recurrence but not an $\epsilon$-vdC set. 
\end{named}
	
\begin{proof}
Let $j\in \N$ and $\epsilon \in \big(0,\frac{1}{2}\big)$. Choose an even 
$Q\in \N$ sufficiently large depending on $j$, 
so that we can apply \cref{ldl1} (with $\ell = 8j$) and \cref{comblem}. 
Since $\epsilon <\frac{1}{2}$ we may choose $Q$ so large that 
\begin{equation*}
\frac{1}{1+\Big(1+C \frac{(8j)^3}{Q^2}\Big)^{\lfloor Q\log 2j^2\rfloor +1}}
>
\epsilon.    
\end{equation*}
Let $P:=\lfloor Q\log 2j^2\rfloor +1$. For each $k\in \{0,1,...,P-1\}$, 
apply \cref{ldl1} with $\ell=8j$ to get a measure $\sigma_k$. 
Consider the non-negative finite measure
$\sigma=\sigma_0\ast \sigma_1 \ast...\ast \sigma_{P-1}$ on $\T$.
For every $k\in \{0,1,...,P-1\}$ we have 
$\widehat{\sigma_k}(0)\leq 1+ C \frac{(8j)^3}{Q^2}$,
and therefore 
\begin{equation*}
0\leq \sigma(\T)
=
\widehat{\sigma}(0)
=
\widehat{\sigma_0}(0)\widehat{\sigma_1}(0)...\widehat{\sigma_{P-1}}(0)
\leq
\Bigg(1+ C \frac{(8j)^3}{Q^2}\Bigg)^P.   
\end{equation*}
Now, consider the measure 
$\mu=\frac{\sigma+\delta_0}{\sigma(\T)+1}$. 
Then 
$\mu(\T)
=
\frac{\sigma(\T)+\delta_0(\T)}{1+\sigma(\T)}
=
1$, i.e. $\mu$ is a probability measure on $\T$. Consider the set 
$S_j:=\big\{r\in \N:\:\widehat{\mu}(r)=0\big\}$. 
Since
\begin{equation*}
\mu(\{0\})
=
\frac{\sigma(\{0\})+1}{1+\sigma(\T)}
\geq 
\frac{1}{1+\sigma(\T)}
\geq 
\frac{1}{1+\Big(1+ C \frac{(8j)^3}{Q^2}\Big)^P}
>
\epsilon,    
\end{equation*}
we have that $S_j$ is not an $\epsilon$-vdC set.\\
\underline{Claim:} $S_j$ is a set of $\big(\frac{1}{j}\big)$-recurrence.
\begin{proof}[Proof of Claim]
Let $E\subseteq \{0,1,...,Q^P-1\}$ with 
$|E|>\frac{Q^P}{j}$. 
Since $P>Q\log 2j^2$, we can apply \cref{comblem} to obtain a 
$y\in E-E$ and an 
$s\in \{0,1,...,P-1\}$ such that when $y$ is written in base $Q$, i.e.
\begin{equation*}
y=\sum_{i=0}^{P-1} y_i Q^i,\:\:y_i\in \{0,1,...,Q-1\},
\end{equation*}
we have that
$\frac{Q}{2}\leq y_s<\frac{Q}{2}+8j$ and for 
$i\neq s$, we have $\:1\leq y_i<8j$. We will show that $y\in S_j$. 
Recall that for each $k\in \{0,1,...,P-1\}$, $\widehat{\sigma_k}$ is periodic
with period $Q^{k+1}$, and therefore
\begin{equation}\label{christmas}
\widehat{\sigma}(y)
=
\prod_{k=0}^{P-1} \widehat{\sigma_k}(y)
=
\prod_{k=0}^{P-1}\widehat{\sigma_k}\Bigg(\sum_{i=0}^{P-1} y_i Q^i\Bigg)
=
\prod_{k=0}^{P-1} \widehat{\sigma_k}\Bigg(\sum_{i=0}^{k} y_i Q^i\Bigg) .
\end{equation}
Recall that $Q$ is sufficiently large depending on $j$, so we may assume that 
$\frac{8j-1}{Q-1}\leq 1$. Then from the hypothesis on $y$ we have that
\begin{equation*}
\frac{Q}{2}
\leq 
\frac{\sum_{i=0}^{s} y_i Q^i}{Q^s}
\leq
\frac{Q}{2}+8j -1  + \frac{(8j-1)}{Q^s} \sum_{i=0}^{s-1} Q^i
\leq
\frac{Q}{2} +8j.    
\end{equation*}
Then, using the properties of $\widehat{\sigma_s}$ (see \cref{ldl1}), we obtain
that $\widehat{\sigma_s}(\sum_{i=0}^{s} y_i Q^i)=-1$. 
On the other hand, for $k\neq s$ we have
\begin{equation*}
1\leq \frac{\sum_{i=0}^{k} y_iQ^i}{Q^k}=y_k 
+
\sum_{i=0}^{k-1} y_i Q^{i-k} \leq 8j,
\end{equation*}
and using the properties of $\widehat{\sigma_k}$ (see \cref{ldl1}), 
we obtain that $\widehat{\sigma_k}(\sum_{i=0}^{k} y_iQ^i)=1$. 
Then, using \eqref{christmas} we obtain $\widehat{\sigma}(y)=-1$ and therefore 
$\widehat{\mu}(y)
=
\frac{\widehat{\sigma}(y) +\widehat{\delta_0}(y)}{\sigma(\T)+1}
=
0$, which means that $y\in S_j$.
As a result, $y\in (E-E)\cap S_j$, which is to say that 
$E\cap (E-y)\neq \emptyset$. So, after all we proved that for any 
$E\subseteq \{0,1,...,Q^P-1\}$ 
with 
$|E|>\frac{Q^P}{j}$, 
there is some $y\in S_j$ such that 
$E\cap (E-y)\neq \emptyset$, 
and therefore $S_j$ is a set of 
$\big(\frac{1}{j}\big)$-recurrence. This concludes the proof of the claim. 
\renewcommand\qedsymbol{$\triangle$}
\end{proof}
		
Finally, from \cref{ldl}, we know that we can find a finite set 
$R_j\subseteq S_j$ such that $R_j$ is again a set of 
$\big(\frac{1}{j}\big)$-recurence. 
Of course, $\widehat{\mu}(y)=0$ for every 
$y\in R_j \:(\subseteq S_j)$ and therefore $R_j$ is not an 
$\epsilon$-vdC set. Hence, $R_j$ is a set with the desired properties, 
and this concludes the proof of the theorem.
\end{proof}

\bibliographystyle{alpha}
\bibliography{refs-andreas}

\bigskip
\footnotesize
\noindent
Andreas Mountakis\\
\textsc{University of Warwick} \par\nopagebreak
\noindent
\href{mailto:andreas.mountakis.1@warwick.ac.uk}
{\texttt{andreas.mountakis.1@warwick.ac.uk}}
\end{document}